\documentclass[12pt,reqno]{amsart}

\usepackage[left=3.5cm,right=3.5cm]{geometry}
\usepackage{amssymb}
\usepackage{graphicx}
\usepackage{amscd}
\usepackage[pagebackref]{hyperref}
\usepackage{color}
\usepackage{tabularx}
\usepackage[table]{xcolor}
\usepackage{float}
\usepackage{graphics,amsmath,amssymb}
\usepackage{amsthm}
\usepackage{amsfonts}
\usepackage{latexsym}
\usepackage{epsf}
\usepackage{xifthen}
\usepackage{mathrsfs}
\usepackage{dsfont}
\usepackage{makecell}
\usepackage{subfig}
\usepackage{amsmath}
\allowdisplaybreaks[4]
\usepackage{listings}
\usepackage{etoolbox}
\usepackage{fancyhdr}
\usepackage{pdflscape}
\usepackage[title,toc,titletoc]{appendix}
\usepackage{enumitem}
\usepackage[noadjust]{cite}
\usepackage{tikz}
\usepackage{young}
\usepackage[object=vectorian]{pgfornament} %%  http://altermundus.com/pages/tkz/ornament/index.html
\usepackage{lipsum,tikz}
\usepackage{multirow}

\hypersetup{
	colorlinks=true, %set true if you want colored links
	linktoc=all, %set to all if you want both sections and subsections linked
	linkcolor=blue} %choose some color if you want links to stand out

\numberwithin{equation}{section}

\theoremstyle{theorem}
\newtheorem{theorem}{Theorem}[section]
\newtheorem*{theorem*}{Theorem}

\newtheorem{corollary}[theorem]{Corollary}

\newtheorem{proposition}[theorem]{Proposition}

\newtheorem{innercustomgeneric}{\customgenericname}
\providecommand{\customgenericname}{}
\newcommand{\newcustomtheorem}[2]{%
	\newenvironment{#1}[1]
	{%
		\renewcommand\customgenericname{#2}%
		\renewcommand\theinnercustomgeneric{##1}%
		\innercustomgeneric
	}
	{\endinnercustomgeneric}
}
\newcustomtheorem{ctheorem}{Theorem}
\newcustomtheorem{clemma}{Lemma}

\theoremstyle{definition}

\newtheorem*{example*}{Example}
\newtheorem*{examples*}{Examples}

\newtheorem*{remark*}{Remark}
\newtheorem*{remarks*}{Remarks}
\newtheorem*{note*}{Note}

\newtheoremstyle{named}{}{}{\itshape}{}{\bfseries}{.}{.5em}{#1\thmnote{ #3}}
\theoremstyle{named}

\newcommand{\card}{\operatorname{card}}

\newcommand{\prob}{\mathbf{P}}

\newcommand{\ri}{\mathrm{i}}
\newcommand{\re}{\mathrm{e}}

\newcommand{\rd}{\operatorname{d}}

\newcommand{\bs}{\boldsymbol{s}}
\newcommand{\bth}{\boldsymbol{\theta}}
\newcommand{\bzero}{\boldsymbol{0}}

\newcommand{\CT}[1]{\underset{#1}{\operatorname{CT}}}

\title[Multi-headed lattices and Green functions]{Multi-headed lattices and Green functions}

\author[Q. Chen]{Qipin Chen}
\address[Q. Chen]{Amazon, Seattle, MA 98109, USA}
\email{qipinche@amazon.com}

\author[S. Chern]{Shane Chern}
\address[S. Chern]{Department of Mathematics and Statistics, Dalhousie University, Halifax, NS, B3H 4R2, Canada}
\email{chenxiaohang92@gmail.com}

\author[L. Jiu]{Lin Jiu}
\address[L. Jiu]{Zu Chongzhi Center for Mathematics and Computational Sciences, Duke Kunshan University, Kunshan, Suzhou, Jiangsu Province, 215316, PR China}
\email{lin.jiu@dukekunshan.edu.cn}

\date{}



\thanks{2010 \textit{PACS Numbers.} 05.50.+q, 05.40.Fb, 02.10.Ox, 02.70.Wz, 05.10.-a}

\begin{document}
	
\maketitle

\sloppy

\begin{abstract}
	Lattice geometries and random walks on them are of great interest for their applications in different fields such as physics, chemistry, and computer science. In this work, we focus on multi-headed lattices and study properties of the Green functions for these lattices such as the associated differential equations and the P\'olya numbers. In particular, we complete the analysis of three missing cases in dimensions no larger than five. Our results are built upon an automatic machinery of creative telescoping.
\end{abstract}

\section{Introduction}

Bravais lattices are important objects in crystallography, and they are used to formally model the orderly arrangement of atoms in a crystal \cite[\S{}42]{LAL1967}. For example, the crystal structure of NaCl can be illustrated by the face-centered cubic Bravais lattice as shown in Fig.~\ref{fig:lattice}. Visually, a \emph{Bravais lattice} is an arrangement of points in the three-dimensional space such that when viewed from each point the lattice appears exactly the same; from a mathematical perspective, it is a $\mathbb{Z}$-module generated by three linearly independent vectors in $\mathbb{R}^3$. Up to equivalence, there are 14 Bravais lattices in the three-dimensional space.

\begin{figure}[ht]
	\captionsetup[subfloat]{position=top,captionskip=20pt,farskip=20pt,margin=5pt,format=hang,singlelinecheck=false,labelfont=rm}
	\centering
	\caption{Crystal structure of NaCl}
	\label{fig:lattice}
	\bigskip
	\begin{minipage}{0.8\textwidth}
		\footnotesize{Plotted via the \textit{Mathematica} package \textsf{Crystallica} implemented by Eifert and Heiliger \cite{EH2016}.}
	\end{minipage}
	
	\subfloat[Crystal structure of NaCl]{\label{fig:NaCa}\includegraphics[height=0.375\textwidth]{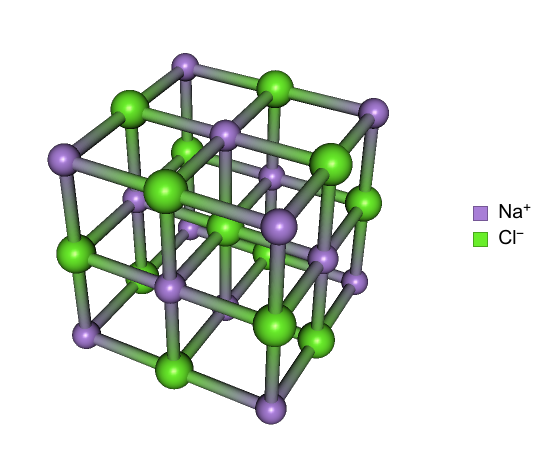}}
	\hspace*{0.035\textwidth}
	\subfloat[Face-centered cubic lattice]{\label{fig:FCC}\includegraphics[height=0.375\textwidth]{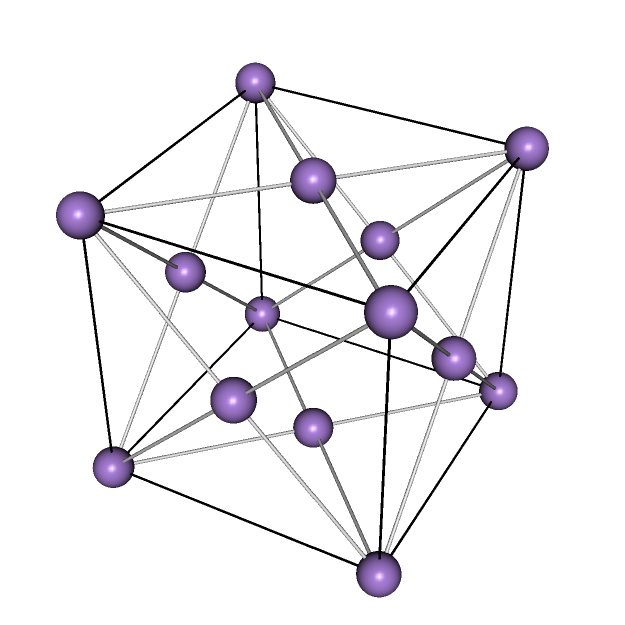}}
\end{figure}

When it comes to higher-dimensional generalizations, if one insists on the criterion of having the same appearance at each lattice point, which is essential in crystal structures, the formalization is usually restricted and sometimes confusing. According to Guttmann \cite[p.~15, Footnote~8]{Gut2010}, such confusion could be eased by adopting the concept of lattice geometries in Meyer's unpublished notes \cite{Mey9999}, which is an addendum to \cite[\S{}1.4]{Mey2000}.

We begin with $\mathbb{Z}^N$, the set of points with integer coordinates in the $N$-dimensional Cartesian space. Let $\mathcal{V}=\{\boldsymbol{v}_1,\ldots,\boldsymbol{v}_q\}$ be a finite set of $N$-dimensional vectors such that for each $\boldsymbol{v}_j$, all its coordinates are taken from $\{-1,0,1\}$. Now a \emph{lattice geometry}, or simply a \emph{lattice}, in dimension $N$ generated by $\mathcal{V}$ is the following subset of $\mathbb{Z}^N$:
\begin{align*}
	\mathcal{L}_{\mathcal{V}}:=\{n_1\boldsymbol{v}_1+\cdots+n_q\boldsymbol{v}_q: n_1,\ldots,n_q\in\mathbb{Z}_{\ge 0}\},
\end{align*}
accompanied with a geometric structure that for two points $\boldsymbol{s}_1,\boldsymbol{s}_2\in \mathcal{L}_{\mathcal{V}}$, there is a \emph{directed} edge from $\boldsymbol{s}_1$ to $\boldsymbol{s}_2$ whenever $\boldsymbol{s}_2-\boldsymbol{s}_1\in \mathcal{V}$. By abuse of notation, we shall also call this lattice $\mathcal{L}_{\mathcal{V}}$. The points in $\mathcal{L}_{\mathcal{V}}$ are called \emph{sites} and the directed edges are called \emph{bonds}. Also, the vectors in $\mathcal{V}$ are called \emph{direction vectors} and the cardinality of $\mathcal{V}$, namely, $q=\card\mathcal{V}$, is called the \emph{coordination number} of this lattice.

It is notable that points in $\mathbb{Z}^N$ are not necessarily sites in a lattice in the same dimension. A simple example is the \emph{face-centered} cubic lattice generated by the 12 direction vectors:
\begin{align*}
	(\pm 1, \pm 1, 0), \quad (\pm 1, 0, \pm 1), \quad (0, \pm 1, \pm 1).
\end{align*}
It is clear that points $(s_1,s_2,s_3)\in\mathbb{Z}^3$ with $s_1+s_2+s_3$ odd are not in the face-centered cubic lattice.

There are many important lattices studied in the literature such as the \emph{simple} and \emph{body-centered} cubic lattices, both of which are Bravais lattices. The simple cubic lattice is generated by $\{(\pm 1,0,0), (0,\pm 1,0), (0,0, \pm 1)\}$ and the body-centered one is generated by $\{(\pm 1, \pm 1, \pm 1)\}$.

Along this line, it becomes natural to consider the following $N$-dimensional lattice generated by the set of direction vectors:
\begin{align}\label{eq:V-MN-def}
	\mathcal{V}_{M,N}:=\big\{(v_1,\ldots,v_N)\in \{-1,0,1\}^N:|v_1|+\cdots+|v_N|=M\big\}.
\end{align}
In other words, the vectors $(v_1,\ldots,v_N)$ are such that $M$ coordinates take value from $\{-1, 1\}$ while the rest are $0$. We shall call this lattice $\mathcal{L}_{\mathcal{V}_{M,N}}$ the \emph{$M$-headed lattice in dimension $N$}, and with this notation, the aforementioned simple, face-centered, and body-centered cubic lattices may be referred to as the $1$-, $2$-, and $3$-headed lattices in dimension three (i.e., $\mathcal{L}_{\mathcal{V}_{1,3}}$, $\mathcal{L}_{\mathcal{V}_{2,3}}$, and $\mathcal{L}_{\mathcal{V}_{3,3}}$), respectively.

As it will be seen later on, the investigation of multi-headed lattices in our current work and elsewhere is truly an Odyssey\footnote{The theme of ``Odyssey'' was famously used by Koutschan, Kauers and Zeilberger \cite{KKZ2011} in the central part of their computer-assisted proof of the Andrews--Robbins $q$-TSPP conjecture. The same technique built upon creative telescoping will be applied in our work as well.}, so we intend to name or rename some lattices of our interest according to Greek mythology. In particular, for the $M=3$ case, we use the name \emph{Cerberus lattices}, and following this nomenclature, we may alternatively call the usual face-centered lattices (i.e., $M=2$) the \emph{Orthrus lattices}. We also name the $M=N-1$ case the \emph{Heracles lattices} for Heracles cutting off the heads, especially the immortal one, of Hydra.

Once we have a lattice geometry, an important topic is about random walks on this lattice, especially noting the fact that we have introduced directed bonds in our lattice construction. In general, random walks have a broad range of applications in areas such as field theory \cite{Bry1984}, quantum chemistry \cite{AF1979}, and web searching \cite{BG2008}.

\begin{figure}[ht]
	\captionsetup[subfloat]{position=top,captionskip=20pt,farskip=20pt,margin=5pt,format=hang,singlelinecheck=false,labelfont=rm}
	\centering
	\caption{Three-dimensional random walks with $1000$ steps}
	\label{fig:random}
	\subfloat[Orthrus lattice\\ Final position at $(-36, 8, 40)$]{\label{fig:random2}\includegraphics[width=0.43\textwidth]{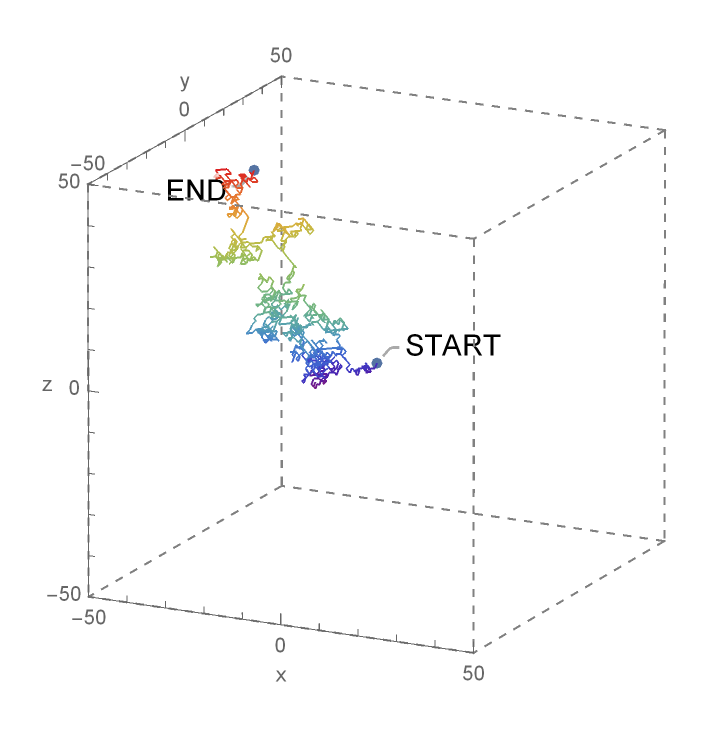}}
	\hspace*{0.035\textwidth}
	\subfloat[Cerberus lattice\\ Final position at $(-8, -42, -12)$]{\label{fig:random3}\includegraphics[width=0.43\textwidth]{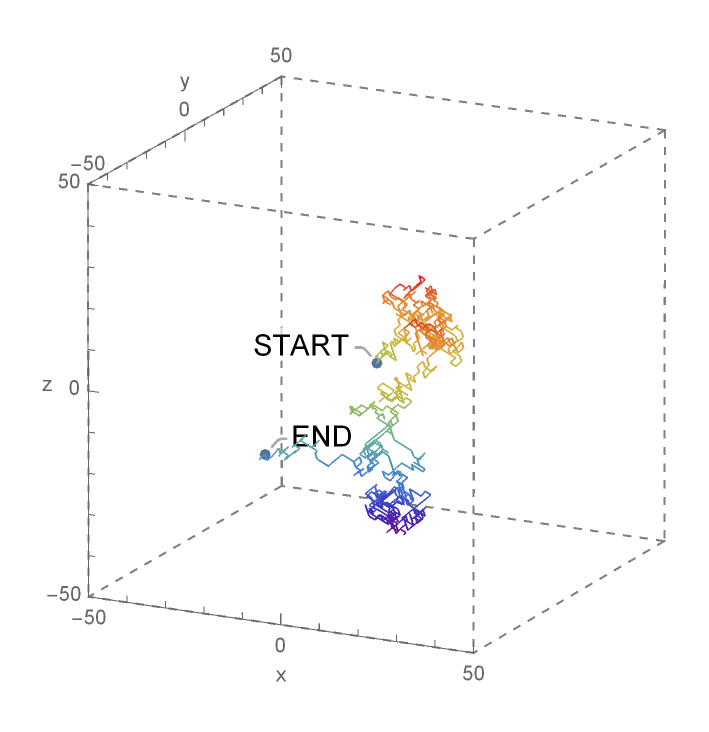}}
\end{figure}

Let $\mathcal{L}_\mathcal{V}$ be a generic $N$-dimensional lattice with $\mathcal{V}$ the set of direction vectors and $q_{\mathcal{L}_\mathcal{V}}$ the coordination number. We say a chain $\mathfrak{w}$ of sites:
\begin{align*}
	\mathfrak{w}:\boldsymbol{0}=\boldsymbol{s}_0\to \boldsymbol{s}_1\to \boldsymbol{s}_2\to \cdots
\end{align*}
is a \emph{random walk} if there is a bond from $\boldsymbol{s}_j$ to $\boldsymbol{s}_{j+1}$ for every $j\ge 0$, and each move from $\boldsymbol{s}_j$ to $\boldsymbol{s}_{j+1}$ is called a \emph{step}. In Fig.~\ref{fig:random}, we have illustrated examples of random walks in three-dimensional Orthrus and Cerberus lattices with $1000$ steps.

A fundamental problem in the theory of random walks is the evaluation of the probability $p_{\mathcal{L}_\mathcal{V}}(\boldsymbol{s},n)$ that a random walk ends at a certain site $\boldsymbol{s}$ after $n$ steps; here we suppose that at each site, the probabilities of moving to the $q_{\mathcal{L}_\mathcal{V}}$ bonded sites as given by $\mathcal{V}$ are equal. We further write the generating function of $p_{\mathcal{L}_\mathcal{V}}(\boldsymbol{s},n)$, which is called the \emph{Green function} of the lattice $\mathcal{L}_\mathcal{V}$, as
\begin{align}
	P_{\mathcal{L}_\mathcal{V}}(\boldsymbol{s},z) := \sum_{n\ge 0} p_{\mathcal{L}_\mathcal{V}}(\boldsymbol{s},n) z^n.
\end{align}
To analyze this generating function, we follow Guttmann \cite{Gut2010} and define the \emph{structure function} of our lattice by
\begin{align}
	\lambda_{\mathcal{L}_\mathcal{V}}(\boldsymbol{\theta}):= \frac{1}{q_{\mathcal{L}_\mathcal{V}}} \sum_{(v_1,\ldots,v_N)\in \mathcal{V}} \re^{\sum_{j=1}^N\ri\,v_j\theta_j},
\end{align}
where $\boldsymbol{\theta}:=(\theta_1,\ldots,\theta_N)$. Then
\begin{align}
	P_{\mathcal{L}_\mathcal{V}}(\boldsymbol{s},z) = \frac{1}{(2\pi)^N}\int_{-\pi}^{\pi}\cdots \int_{-\pi}^{\pi} \frac{\re^{-\ri\,\bs\cdot \bth}\rd \theta_1\cdots \rd \theta_N}{1-z\lambda_{\mathcal{L}_\mathcal{V}}(\bth)}.
\end{align}
Of particular interest is the scenario that a random walk returns to the origin in which case it is called an \emph{excursion}. Namely, we evaluate the lattice Green function at the site $\bs = \bzero$:
\begin{align}
	P_{\mathcal{L}_\mathcal{V}}(\boldsymbol{0},z) = \frac{1}{(2\pi)^N}\int_{-\pi}^{\pi}\cdots \int_{-\pi}^{\pi} \frac{\rd \theta_1\cdots \rd \theta_N}{1-z\lambda_{\mathcal{L}_\mathcal{V}}(\bth)}.
\end{align} 
This function gives the \emph{P\'olya number} $\prob_{\mathcal{L}_\mathcal{V}}(\boldsymbol{0})$ of the lattice in question, that is, the probability that a random walk will eventually return to its original site. In light of \emph{P\'olya's random walk theorem} \cite{Pol1921},
\begin{align}
	\prob_{\mathcal{L}_\mathcal{V}}(\boldsymbol{0}) = 1-\frac{1}{P_{\mathcal{L}_\mathcal{V}}(\boldsymbol{0},1)}.
\end{align}

Green functions for lattices, especially in lower dimensions, have been widely analyzed in the literature. In the case of three dimensions, one may find analytic properties of the Green functions for simple and face-centered cubic lattices in the works of Glasser and Zucker \cite{GZ1987} and Joyce \cite{Joy1994,Joy1998}, and these Green functions are connected with elliptic integrals; see Glasser \cite{Gla2006} for some comprehensive insights. In the four-dimensional circumstance, progress on the simple lattice was made by Glasser, Guttmann, and their collaborators \cite{GG1994,GM1993,GP1993}, and Guttmann \cite{Gut2009} also considered the face-centered case. For face-centered lattices in higher dimensions, we have some recent results in \cite{Bro2009,HKMZ2016,Kou2013,ZHM2015}. For lattices that are not multi-headed, we also witness the contributions of Joyce on the diamond lattice \cite{Joy1973} and the honeycomb lattice \cite{Joy2018,Joy2019}. Almost all of the above were nicely summarized in Guttmann's survey \cite{Gut2010}.

For multi-headed lattices in dimensions no larger than five, there are only three missing cases, namely the Heracles lattices ($M=N-1$) in dimensions $N=4$ and $5$ and the Cerberus lattice ($M=3$) in dimension $N=5$. Hence, the central objective of this work is to analyze some basic properties of the associated Green functions, thereby completing the gaps. We will perform our study \`a la Koutschan \cite{Kou2013}, who took advantage of the creative telescoping method popularized by Zeilberger \cite{Zei1990} that was later implemented in an automatic way by Koutschan himself \cite{Kou2010a} in his \textit{Mathematica} package \textsf{HolonomicFunctions} \cite{Kou2010b}. It is notable that our effort on the Heracles lattices, although still built upon creative telescoping, relies on a summation expression related to the associated Green functions, and this reformulation offers more computational efficiency compared with Koutschan's original work on an alternative integral expression.

\begin{note*}
	All \textit{Mathematica} notebooks for the computations in Subsects.~\ref{sec:H-4D}, \ref{sec:H-5D} and \ref{sec:C-5D}, and Sect.~\ref{sec:conclusion} can be found in \cite{CCJ2024}.
\end{note*}

\section{Multi-headed lattices}

In this section, we work out some general properties of multi-headed lattices $\mathcal{L}_{\mathcal{V}_{M,N}}$. For convenience, we write the Green function $P_{\mathcal{L}_{\mathcal{V}_{M,N}}}(\boldsymbol{s},z)$ as $P_{M,N}(\boldsymbol{s},z)$, and make the same change of subscripts for other functions or parameters accordingly.

In light of \eqref{eq:V-MN-def}, we may first rewrite the structure function as
\begin{align*}
	\lambda(\boldsymbol{\theta})=\lambda_{M,N}(\boldsymbol{\theta}) &= \frac{1}{q} \sum_{(v_1,\ldots,v_N)\in \mathcal{V}_{M,N}} \re^{\sum_{j=1}^N\ri\,v_j\theta_j}\\
	&= \frac{1}{q} \sum_{1\le j_1<\cdots<j_M\le N}\big(\re^{\ri\,\theta_{j_1}}+\re^{-\ri\,\theta_{j_1}}\big)\cdots \big(\re^{\ri\,\theta_{j_M}}+\re^{-\ri\,\theta_{j_M}}\big),
\end{align*}
where
\begin{align}
	q=q_{M,N}=\card \mathcal{V}_{M,N} = 2^M \binom{N}{M}.
\end{align}
Let $\sigma_M(x_1,\ldots,x_N)$ be the $M$-th \emph{elementary symmetric polynomial} in $N$ variables:
\begin{align*}
	\sigma_M(x_1,\ldots,x_N) := \sum_{1\le j_1<\cdots<j_M\le N} x_{j_1}\cdots x_{j_M}.
\end{align*}
If we write
\begin{align}\label{eq:X-Y-def}
	X_j := \re^{\ri\, \theta_j},\qquad Y_j := X_j+X_j^{-1},
\end{align}
then
\begin{align}
	\lambda(\boldsymbol{\theta}) &= \frac{1}{2^M \binom{N}{M}} \sigma_M (Y_1,\ldots,Y_N).
\end{align}
The above also simplifies to
\begin{align}
	\lambda(\boldsymbol{\theta})= \frac{1}{\binom{N}{M}} \sigma_M (\cos\theta_1,\ldots,\cos\theta_N).
\end{align}

Therefore, the Green function for excursions is
\begin{align}
	P_{M,N}(\boldsymbol{0},z) = \frac{1}{(2\pi)^N}\int_{-\pi}^{\pi}\cdots \int_{-\pi}^{\pi} \frac{\rd \theta_1\cdots \rd \theta_N}{1-\frac{z}{\binom{N}{M}}\sigma_M(\cos\theta_1,\ldots,\cos\theta_N)}.
\end{align} 
For convenience, we further normalize the above as
\begin{align}\label{eq:R-int}
	R_{M,N}(\boldsymbol{0},z)&:= P_{M,N}\big(\boldsymbol{0},2^M\tbinom{N}{M}z\big)\notag\\
	&\;= \frac{1}{(2\pi)^N}\int_{-\pi}^{\pi}\cdots \int_{-\pi}^{\pi} \frac{\rd \theta_1\cdots \rd \theta_N}{1-z\, 2^M\sigma_M(\cos\theta_1,\ldots,\cos\theta_N)},
\end{align} 
and write
\begin{align}\label{eq:r-MN}
	R_{M,N}(\boldsymbol{0},z) = \sum_{n\ge 0} r_{M,N}(n) z^n.
\end{align}
Recalling the change of variables in \eqref{eq:X-Y-def}, we have
\begin{align}\label{eq:R-oint}
	R_{M,N}(\boldsymbol{0},z) = \frac{1}{(2\pi\ri)^N}\oint_{\mathcal{C}}\frac{\rd X_1}{X_1}\cdots \oint_{\mathcal{C}}\frac{\rd X_N}{X_N} \frac{1}{1-z\, \sigma_M (Y_1,\ldots,Y_N)},
\end{align}
where $\mathcal{C}$ is the unit circle oriented counterclockwise.

Given a Laurent polynomial $f\in \mathbb{C}[x_1,x_1^{-1},\ldots,x_N,x_N^{-1}]$, we denote by $\CT{x_1,\ldots,x_N}f$ its \emph{constant term}, that is, the coefficient of $x_1^0\cdots x_N^0$. It is clear from \eqref{eq:R-oint} that
\begin{align}\label{eq:p-CT}
	r_{M,N}(n) = \CT{X_1,\ldots,X_N}\big(\sigma_M (Y_1,\ldots,Y_N)\big)^n.
\end{align}
In addition, for $Y=X+X^{-1}$,
\begin{align}\label{eq:CT-Y}
	\CT{X}Y^d = \CT{X}\big(X+X^{-1}\big)^d=\begin{cases}
		\scalebox{0.7}{$\dbinom{d}{\frac{d}{2}}$}, & \text{if $d$ is even},\\[6pt]
		0, & \text{if $d$ is odd}.
	\end{cases}
\end{align}

\begin{proposition}
	We have $r_{M,N}(2n+1)=0$ for all $n\ge 0$ if $M$ is odd or $M=N$.
\end{proposition}

\begin{proof}
	We first consider the case where $M$ is odd. Expanding the power $\big(\sigma_{M} (Y_1,\ldots,Y_N)\big)^{2n+1}$, it is clear that all its terms are of the form $Y_1^{d_1}Y_2^{d_2}\cdots Y_N^{d_N}$ with
	\begin{align*}
		d_1+d_2+\cdots +d_N = (2n+1)M
	\end{align*}
	being an odd integer. Hence at least one of the $d$'s, say $d_j$, is also odd. Then by \eqref{eq:CT-Y}, $\CT{X_j}Y_j^{d_j} = 0$ so that
	\begin{align*}
		\CT{X_1,\ldots,X_N}Y_1^{d_1}Y_2^{d_2}\cdots Y_N^{d_N} = \left(\CT{X_1}Y_1^{d_1}\right)\left(\CT{X_2}Y_2^{d_2}\right)\cdots \left(\CT{X_N}Y_N^{d_N}\right) = 0.
	\end{align*}
	Since all terms in the expansion of $\CT{X_1,\ldots,X_N}\big(\sigma_{M} (Y_1,\ldots,Y_N)\big)^{2n+1}$ vanish, so does itself. Recalling \eqref{eq:p-CT} gives our result.
	
	Next, for the case where $M=N$, we have
	\begin{align*}
		r_{N,N}(n) = \CT{X_1,\ldots,X_N}Y_1^n\cdots Y_N^n = \left(\CT{X}Y^{n}\right)^N.
	\end{align*}
	In view of \eqref{eq:CT-Y}, we also have the vanishing of $r_{N,N}(2n+1) = 0$. It is also notable that for even arguments,
	\begin{align*}
		r_{N,N}(2n) = \binom{2n}{n}^N.
	\end{align*}
	This evaluation is standard and can be found in Guttmann's survey \cite[Sect.~2.1]{Gut2010}.
\end{proof}

The above observation implies that for $M$ odd or $M=N$,
\begin{align}\label{eq:R-M-odd}
	R_{M,N}(\boldsymbol{0},z) = \sum_{n\ge 0} r_{M,N}(2n) z^{2n}.
\end{align}
Hence, for these cases, we make the normalization
\begin{align}\label{eq:P-normalization}
	\widetilde{R}_{M,N}(\boldsymbol{0},z) := \sum_{n\ge 0} r_{M,N}(2n) z^{n},
\end{align}
and write
\begin{align}\label{eq:p-normalization}
	\widetilde{r}_{M,N}(n) := r_{M,N}(2n).
\end{align}

\section{``Differential Equations'' versus ``Recurrences''}\label{sec:ODE-REC}

Staying in the heart of analytic properties of Green functions for multi-headed lattices in previous contributions are the differential equations satisfied by these Green functions. In this section, we provide some preliminary information.

A formal power series $F=F(z)$ over a field $\mathbb{K}$ is \emph{D-finite} if there are polynomials $v_0(z),\ldots,v_K(z)\in \mathbb{K}[z]$ with $v_K(z)$ not identical to zero such that
\begin{align*}
	\big(v_0(z)+v_1(z)D+\cdots+ v_K(z)D^K\big)(F) = 0,
\end{align*}
where $D=D_z:=\frac{\rd}{\rd z}$ is the usual \emph{derivation operator}. We may alternatively utilize the \emph{Euler operator} $\vartheta=\vartheta_z$ defined by
\begin{align}
	\vartheta=\vartheta_z:= z\frac{\rd}{\rd z},
\end{align}
and find another family of polynomials $u_0(z),\ldots,u_K(z)\in \mathbb{K}[z]$ with $u_K(z)$ not identical to zero for the D-finite series $F$ such that
\begin{align}\label{eq:ODE-def}
	\big(u_0(z)+u_1(z)\vartheta+\cdots+ u_K(z)\vartheta^K\big)(F) = 0.
\end{align}
Now we say $F$ satisfies the differential equation (with respect to $\vartheta$) $u_0(z)+u_1(z)\vartheta+\cdots+ u_K(z)\vartheta^K$, and call $K$ the \emph{order} and the maximal degree $L$ of the polynomials $u_k(z)$ the \emph{degree} of this differential equation.

Our main objective is to find a differential equation (with respect to $\vartheta$) satisfied by the Green functions for multi-headed lattices. Before exploring this topic, we provide some discussions on the relations between \emph{differential equations} satisfied by a formal power series $F$ and \emph{recurrences} satisfied by the coefficients of this series. Writing $F(z)=\sum_{n\ge 0} f(n)z^n$, we note that for $k\ge 0$,
\begin{align*}
	\vartheta^k (F) = \sum_{n\ge 0} f(n) \big(\vartheta^k z^n\big) = \sum_{n\ge 0} n^k f(n) \cdot z^n.
\end{align*}
Hence, if we reformulate \eqref{eq:ODE-def} as
\begin{align*}
	\left(\sum_{k=0}^K \sum_{\ell=0}^L u_{k,\ell} z^\ell \vartheta^k\right)(F) = 0
\end{align*}
with $u_{k,\ell}\in \mathbb{K}$, then we have a recurrence for the coefficients $f(n)$ in $F(z)$:
\begin{align}\label{eq:ODE-REC}
	\sum_{\ell=0}^L \sum_{k=0}^K u_{k,\ell} (n-\ell)^k f(n-\ell) = 0.
\end{align}
It turns out that the order of the recurrence satisfied by $f(n)$ is at most $L$.

In general, we say a sequence $\{s(n)\}_{n\ge 0}$ is \emph{P-finite} of order $L$ if there are polynomials $h_0(n),\ldots, h_L(n)$ with $h_0(n)$ and $h_L(n)$ not identical to zero such that the following recurrence hold:
\begin{align}\label{eq:P-finite-def}
	h_0(n)s(n)+h_1(n)s(n-1)+\cdots h_L(n)s(n-L) = 0.
\end{align}
Clearly, such a recurrence also produces a differential equation for the generating function of $s(n)$.

In the past study on the Green functions, it was usually expected that the orders of the associated differential equations could be as small as possible. This is mainly due to the connection between these differential equations and the so-called Calabi--Yau equations \cite{AvEvSZ}. Meanwhile, from a computational perspective, a recurrence of a smaller order usually possesses benefits when generating a long list of values of the sequence in question. Hence, in our work, we intend to find differential equations for the Green functions such that either the \emph{order} or the \emph{degree} is possibly minimal.

Now we briefly introduce our procedure:

Firstly, we follow Koutschan \cite{Kou2013} and apply the method of creative telescoping to get a differential equation $\mathtt{ODE}$ for our normalized Green functions $R(z)$, or alternatively $\widetilde{R}(z)$ for the cases associated with \eqref{eq:P-normalization}. This differential equation then gives us a recurrence, say $\mathtt{REC}$, for the coefficients $r(n)$ in the series expansion of $R(z)$ in view of \eqref{eq:ODE-REC}. Meanwhile, the initial values of $r(n)$ can be computed by \eqref{eq:p-CT} or by other means. Combining these initial values with the recurrence $\mathtt{REC}$, we are then able to generate a list of $r(n)$ of a considerable length.

Next, we take advantage of the \textit{Mathematica} package \textsf{Guess} implemented by Kauers \cite{Kau2009}. There are two options:
\begin{enumerate}[label={\textup{(\arabic*)~}},leftmargin=*,labelsep=0cm,align=left,itemsep=6pt]
	\item The command \textbf{GuessMinDE} allows us to conjecture a differential equation $\mathtt{ODE}'$ of possibly the minimal order for $R(z)$ from the list of values of $r(n)$. By \eqref{eq:ODE-REC}, this new differential equation implies a recurrence $\mathtt{REC}'$ for $r(n)$.
	
	\item The command \textbf{GuessMinRE} allows us to conjecture a recurrence $\mathtt{REC}''$ of possibly the minimal order for $r(n)$. This recurrence equivalently yields a differential equation $\mathtt{ODE}''$ for $R(z)$ which is of possibly the minimal degree.
\end{enumerate}

Finally, we are left to show that the two conjectural recurrences $\mathtt{REC}'$ and $\mathtt{REC}''$ are satisfied by $r(n)$. Let us temporarily denote by $r'(n)$ (resp.~$r''(n)$) the sequence generated by the recurrence $\mathtt{REC}'$ (resp.~$\mathtt{REC}''$) and the values of $r(n)$ up to the order of $\mathtt{REC}'$ (resp.~$\mathtt{REC}''$). Then the sequences $r(n)$, $r'(n)$ and $r''(n)$ are P-finite. Recall that for two P-finite sequences, their linear combinations are also P-finite \cite[Theorem 4]{Kau2013}. Hence, there exist recurrences $\mathtt{REC}^*$ and $\mathtt{REC}^{**}$ satisfied by $r(n)-r'(n)$ and $r(n)-r''(n)$, respectively. In particular, $\mathtt{REC}^*$ and $\mathtt{REC}^{**}$ can be computed by utilizing the \textbf{DFinitePlus} command of Koutschan's \textit{Mathematica} package \textsf{HolonomicFunctions} \cite{Kou2010b}. Hence, as long as we verify that $r(n)-r'(n)=0$ (resp.~$r(n)-r''(n)=0$) for $n$ up to the order of $\mathtt{REC}^*$ (resp.~$\mathtt{REC}^{**}$), it is safe to conclude that $r(n)-r'(n)=0$ (resp.~$r(n)-r''(n)=0$) for all $n$, thereby implying that $r(n)$ also satisfies the recurrence $\mathtt{REC}'$ (resp.~$\mathtt{REC}''$) as claimed.

\section{Heracles lattices}

Recall that the Heracles lattice in dimension $N$ is such that $M=N-1$. In this section, we first prove a general summation expression for the coefficients $r_{N-1,N}(n)$ given by \eqref{eq:r-MN}. Then we discuss some benefits of this summation expression in our application of creative telescoping. Finally, we look into the two particular cases $N=4$ and $5$ and perform some calculations on the associated Green functions.

\subsection{A summation formula}

We prove the following result.

\begin{theorem}\label{th:H-sum}
	For $n\ge 0$,
	\begin{align}\label{eq:M=N-1:even}
		r_{N-1,N}(2n)=\sum_{\tiny\substack{k_1,\ldots,k_N\ge 0\\k_1+\cdots+k_N=n}}\scalebox{0.8}{$\dbinom{2n}{2k_1,\ldots,2k_N}\dbinom{2(n-k_1)}{n-k_1}\cdots\dbinom{2(n-k_N)}{n-k_N}$}.
	\end{align}
	Also, if $N$ is even,
	\begin{align}\label{eq:M=N-1:odd1}
		r_{N-1,N}(2n+1) = 0;
	\end{align}
	if $N$ is odd,
	\begin{align}\label{eq:M=N-1:odd2}
		r_{N-1,N}(2n+1)=\sum_{\tiny\substack{k_1,\ldots,k_N\ge 0\\k_1+\cdots+k_N=n+\frac{1-N}{2}}}\scalebox{0.8}{$\dbinom{2n+1}{2k_1+1,\ldots,2k_N+1}\dbinom{2(n-k_1)}{n-k_1}\cdots\dbinom{2(n-k_N)}{n-k_N}$}.
	\end{align}
\end{theorem}

\begin{proof}
	We start by noting that
	\begin{align*}
		\sigma_{N-1} (Y_1,\ldots,Y_N) = Y_1\cdots Y_N \big(Y_1^{-1}+\cdots+Y_N^{-1}\big).
	\end{align*}
	Hence,
	\begin{align*}
		\big(\sigma_{N-1} (Y_1,\ldots,Y_N)\big)^n &= Y_1^n\cdots Y_N^n \big(Y_1^{-1}+\cdots+Y_N^{-1}\big)^n\\
		&= \sum_{\substack{j_1,\ldots,j_N\ge 0\\j_1+\cdots+j_N=n}}\binom{n}{j_1,\ldots,j_N}Y_1^{n-j_1}\cdots Y_N^{n-j_N}.
	\end{align*}
	It follows from \eqref{eq:CT-Y} that
	\begin{align*}
		&\CT{X_1,\ldots,X_N}\big(\sigma_{N-1} (Y_1,\ldots,Y_N)\big)^n\\
		&\qquad= \sum_{\substack{j_1,\ldots,j_N\ge 0\\j_1+\cdots+j_N=n}}\dbinom{n}{j_1,\ldots,j_N}\left(\CT{X_1}Y_1^{n-j_1}\right)\cdots \left(\CT{X_N}Y_N^{n-j_N}\right)\\
		&\qquad= \sum_{\substack{j_1,\ldots,j_N\ge 0\\j_1+\cdots+j_N=n\\j_1\equiv\cdots \equiv j_N\equiv n \bmod{2}}}\dbinom{n}{j_1,\ldots,j_N}\binom{n-j_1}{\frac{n-j_1}{2}}\cdots \binom{n-j_N}{\frac{n-j_N}{2}},
	\end{align*}
	which implies the claimed result in view of \eqref{eq:p-CT}.
\end{proof}

\subsection{``Summation'' versus ``Integral''}\label{sec:sum-int}

We recall that in \cite{Kau2013}, Koutschan first made the change of variables:
\begin{align*}
	x_j:= \cos\theta_j
\end{align*}
so as to obtain
\begin{align*}
	R_{M,N}(\boldsymbol{0},z)= \frac{1}{\pi^N}\int_{-1}^{1}\cdots \int_{-1}^{1} \frac{\rd x_1\cdots \rd x_N}{\big(1-z\, 2^M\sigma_M(x_1,\ldots,x_N)\big)\sqrt{1-x_1^2}\cdots \sqrt{1-x_N^2}}.
\end{align*} 
He then applied creative telescoping to the integral \eqref{eq:R-int}.

Now since we have an alternative summation expression for $r_{N-1,N}(n)$, creative telescoping can be applied to it directly. The case where $N-1$ is odd is easier as for such $N$, we have the vanishing of $r_{N-1,N}(2n+1)$ by \eqref{eq:M=N-1:odd1}. Hence, it is sufficient to work on the normalization \eqref{eq:p-normalization}:
\begin{align*}
	\widetilde{r}_{N-1,N}(n) := r_{N-1,N}(2n).
\end{align*}
It is somehow tricky to handle the case where $N-1$ is even. This is because we have to work on the even- and odd-indexed subsequences separately:
\begin{align*}
	\widetilde{r}^e_{N-1,N}(n) &:= r_{N-1,N}(2n),\\
	\widetilde{r}^o_{N-1,N}(n) &:= r_{N-1,N}(2n+1).
\end{align*}
Once we obtain recurrences for $\widetilde{r}^e_{N-1,N}(n)$ and $\widetilde{r}^o_{N-1,N}(n)$ respectively by creative telescoping, we equivalently have recurrences for the sequences:
\begin{align*}
	\widetilde{r}^{e,*}_{N-1,N}&:=\big\{r_{N-1,N}(0),0,r_{N-1,N}(2),0,r_{N-1,N}(4),0,\ldots\big\},\\
	\widetilde{r}^{o,*}_{N-1,N}&:=\big\{0,r_{N-1,N}(1),0,r_{N-1,N}(3),0,r_{N-1,N}(5),\ldots\big\}.
\end{align*}
Now we can use Koutschan's \textbf{DFinitePlus} to get a recurrence for any linear combination of the above two sequences, including
\begin{align*}
	\widetilde{r}^{e,*}_{N-1,N}+\widetilde{r}^{o,*}_{N-1,N}&=\big\{r_{N-1,N}(0),r_{N-1,N}(1),r_{N-1,N}(2),r_{N-1,N}(3),\ldots\big\}\\
	&= r_{N-1,N}.
\end{align*}
It is notable that the order of the recurrence derived from the above computation is usually larger than that of the minimal recurrence for $r_{N-1,N}(n)$, and in the meantime the order of the differential equation for $R_{N-1,N}(z)$ associated with this recurrence is also not minimal. Hence, we have to perform the ``\emph{Guessing}'' procedure presented in the last part of Sect.~\ref{sec:ODE-REC}.

\subsection{Case: 4D}\label{sec:H-4D}

Note that in this case $M=3$ is odd, so we look into $\widetilde{r}_{3,4}(n)$ and accordingly $\widetilde{R}_{3,4}(\boldsymbol{0},z)$. We first mimic Koutschan's work \cite{Kou2013} and directly apply creative telescoping to the integral
\begin{align*}
	R_{3,4}(\boldsymbol{0},z)= \frac{1}{\pi^4}\int_{-1}^{1}\int_{-1}^{1}\int_{-1}^{1} \int_{-1}^{1} \texttt{(***)}\rd x_1\rd x_2\rd x_3 \rd x_4,
\end{align*}
where the integrand is
\begin{align*}
	\frac{1}{\big(1-8z(x_1x_2x_3+x_2x_3x_4+x_3x_4x_1+x_4x_1x_2)\big)\sqrt{1-x_1^2}\sqrt{1-x_2^2}\sqrt{1-x_3^2} \sqrt{1-x_4^2}}.
\end{align*}
After around \emph{one hour} of computation, we arrive at the following result.

\begin{theorem}\label{th:M3N4}
	The function $\widetilde{R}_{3,4}(\boldsymbol{0},z)$ satisfies a differential equation of order $8$ and degree $16$. The differential equation is of the form
	\begin{align*}
		\scalebox{0.85}{%
			$
			\begin{aligned}
				&(1344 z + \cdots + 151026323282253922352374256330569782279536640 z^{16})\\
				&+ (18816 z + \cdots + 1366788225704397997288987019791656529629806592 z^{16})\vartheta\\
				&+(102144 z + \cdots + 5303121535030477312378786039652035049432285184 z^{16})\vartheta^2\\
				&+ (212352 z + \cdots + 11533376887988124536976314041777845706747281408 z^{16})\vartheta^3\\
				&+(-42 + \cdots + 15391679930285039325517386362534096505705332736 z^{16})\vartheta^4\\
				&+ (357 + \cdots + 12917784851408785491873078058141402044309700608 z^{16})\vartheta^5\\
				&+(-1113 + \cdots + 6663616997264781396236424132096584504800444416 z^{16})\vartheta^6\\
				&+ (1512 + \cdots + 1933136938012850206110390481031293213178068992 z^{16})\vartheta^7\\
				&+(-756 + \cdots + 241642117251606275763798810128911651647258624 z^{16})\vartheta^8.
			\end{aligned}
			$
		}
	\end{align*}
	Accordingly, $\widetilde{r}_{3,4}(n)$ satisfies a recurrence of order $16$.
\end{theorem}

Meanwhile, if we directly apply creative telescoping to the summation
\begin{align*}
	\widetilde{r}_{3,4}(n) = \sum_{\tiny\substack{k_1,k_2,k_3,k_4\ge 0\\k_1+k_2+k_3+k_4=n}}\scalebox{0.8}{$\dbinom{2n}{2k_1,2k_2,2k_3,2k_4}\dbinom{2(n-k_1)}{n-k_1}\dbinom{2(n-k_2)}{n-k_2}\dbinom{2(n-k_3)}{n-k_3}\dbinom{2(n-k_4)}{n-k_4}$},
\end{align*}
we get the following recurrence in ONLY \emph{ten minutes}! This fact to some extent indicates that our summation expression is superior in a computational sense.

\begin{theorem}\label{th:M3N4-rec}
	The sequence $\widetilde{r}_{3,4}(n)$ satisfies a recurrence of order $4$. The recurrence is of the form
	\begin{align*}
		\scalebox{0.85}{%
			$
			\begin{aligned}
				0&=(221086792032258663383040 + \cdots + 1988330027074191360 n^{20})\widetilde{r}_{3,4}(n)\\
				&- (123596648884357621088256 + \cdots + 135920997944524800 n^{20})\widetilde{r}_{3,4}(n+1)\\
				&+(2413729498666800513024 + \cdots + 745214176788480 n^{20})\widetilde{r}_{3,4}(n+2)\\
				&- (9569617440812835840 + \cdots + 1074030451200 n^{20})\widetilde{r}_{3,4}(n+3)\\
				&+(9051531325562880 + \cdots + 462944160 n^{20})\widetilde{r}_{3,4}(n+4).
			\end{aligned}
			$
		}
	\end{align*}
	Accordingly, $\widetilde{R}_{3,4}(\boldsymbol{0},z)$ satisfies a differential equation of order $24$ and degree $4$.
\end{theorem}

By utilizing the \textsf{Guess} package of Kauers to predict the minimal orders of the differential equation for $\widetilde{R}_{3,4}(\boldsymbol{0},z)$ and the recurrence for $\widetilde{r}_{3,4}(n)$, it seems that those in Theorems \ref{th:M3N4} and \ref{th:M3N4-rec} are the minimal ones.

Finally, following a heuristic idea of Wimp and Zeilberger \cite{WZ1985}, it is usually possible to evaluate the asymptotic expansion of a given P-finite sequence $s(n)$. Briefly speaking, we can make an ansatz that $s(n)$ can be asymptotically written as a linear combination of the expressions:
\begin{align*}
	n^{\alpha_1 n+\alpha_0}\,\re^{\sum_{j=1}^{r-1}\mu_j n^{\frac{j}{r}}} \rho^n (\log n)^{\nu}\left(1+\sum_{k\ge 1}\beta_{k} n^{-k}\right),
\end{align*}
which are asymptotic solutions to the recurrence relation for $s(n)$ so that the parameters in the above can be explicitly evaluated. Then the scalars in the linear combination can be calculated numerically by the values of $s(n)$ for large $n$. Such an analysis was efficiently implemented by Kauers in his \textit{Mathematica} package \textsf{Asymptotics} \cite{Kau2011}.

\begin{corollary}
	As $n\to \infty$,
	\begin{align}\label{eq:p34}
		\widetilde{r}_{3,4}(n) = r_{3,4}(2n) \sim \widetilde{C}_{3,4}\cdot 1024^n n^{-2},
	\end{align}
	where the constant $\widetilde{C}_{3,4}\approx 0.0225$.
\end{corollary}

From \eqref{eq:p34}, we can control the tail of $P_{3,4}(\boldsymbol{0},1)$ by
\begin{align*}
	\sum_{n\ge 2n_0} r_{3,4}(n)\left(\frac{1}{2^3\cdot \binom{4}{3}}\right)^n = \sum_{n\ge n_0} r_{3,4}(2n)\left(\frac{1}{32}\right)^{2n} \asymp n_0^{-1},
\end{align*}
when $n_0$ is large enough. Hence, evaluating the first thousands of terms in $P_{3,4}(\boldsymbol{0},1)$ would be sufficient to give a nice estimation $P_{3,4}(\boldsymbol{0},1) \approx 1.04528$. The P\'olya number could then be computed accordingly.

\begin{corollary}
	The P\'olya number of the four-dimensional Heracles lattice is approximately
	\begin{align}
		\prob_{3,4}(\boldsymbol{0}) \approx 0.04332.
	\end{align}
\end{corollary}

\subsection{Case: 5D}\label{sec:H-5D}

Note that in this case $M=4$ is even, so we look into $r_{4,5}(n)$ and accordingly $R_{4,5}(\boldsymbol{0},z)$. Here we only apply creative telescoping to our summation expression; this process took us around \emph{two days}. We also note that the \textit{Mathematica} code for the integral expression had not yet finished running after even around \emph{one week}! So we decided to quit the kernel.

We find that the recurrence for $r_{4,5}(n)$ deduced from applying creative telescoping to the summation expression has order $12$. Also, the corresponding differential equation for $R_{4,5}(\boldsymbol{0},z)$ has order $83$. It is clear that they are not minimal by consulting the \textsf{Guess} package and this observation is consistent with the discussion at the end of Sect.~\ref{sec:sum-int}. Hence, invoking the procedure in the last part of Sect.~\ref{sec:ODE-REC} becomes necessary.

The main results are recorded as follows.

\begin{theorem}\label{th:M4N5-ODE}
	The function $R_{4,5}(\boldsymbol{0},z)$ satisfies a differential equation of order $9$ and degree $24$. The differential equation is of the form
	\begin{align*}
		\scalebox{0.85}{%
			$
			\begin{aligned}
				&(333047697408 z + \cdots + 3587135914162316664577589182300422144000000000000 z^{24})\\
				&+ (47239200 + \cdots + 8746714696058467589996180287043036774400000000000 z^{24})\vartheta\\
				&+(291308400 + \cdots + 9453516646138192392531605664037066506240000000000 z^{24})\vartheta^2\\
				&+ (740080800 + \cdots + 5943864920522147046229253977370938834944000000000 z^{24})\vartheta^3\\
				&+(1027452600 + \cdots + 2395763624685018201440807167653926928384000000000 z^{24})\vartheta^4\\
				&+ (861131250 + \cdots + 641926999294401201636284114667753701376000000000 z^{24})\vartheta^5\\
				&+(449264475 + \cdots + 114331247240822245206661000977438474240000000000 z^{24})\vartheta^6\\
				&+ (143193825 + \cdots + 13051386184451845257422171891507920896000000000 z^{24})\vartheta^7\\
				&+(25587900 + \cdots + 866423326586435904331316912053026816000000000 z^{24})\vartheta^8\\
				&+(1968300 + \cdots + 25483039017248114833274026825089024000000000 z^{24})\vartheta^9.
			\end{aligned}
			$
		}
	\end{align*}
	Accordingly, $r_{4,5}(n)$ satisfies a recurrence of order $24$.
\end{theorem}

\begin{theorem}\label{th:M4N5}
	The sequence $r_{4,5}(n)$ satisfies a recurrence of order $6$. The recurrence is of the form
	\begin{align*}
		\scalebox{0.85}{%
			$
			\begin{aligned}
				0&=(2364822061925891270067722649600000 + \cdots + 312808771118086225920 n^{27})r_{4,5}(n)\\
				&+ (880540948213763261498004602880000 + \cdots + 22881382331785936896 n^{27})r_{4,5}(n+1)\\
				&-(664078540666702251488371015680000 + \cdots + 5976795675008958464 n^{27})r_{4,5}(n+2)\\
				&+ (36337840931616555318702833664000 + \cdots + 159149910074064896 n^{27})r_{4,5}(n+3)\\
				&+(1737772868400007324872130560000 + \cdots + 3900964176134144 n^{27})r_{4,5}(n+4)\\
				&-(36446102109669030849285120000 + \cdots + 51561082388480 n^{27})r_{4,5}(n+5)\\
				&-(154404486709237819219968000 + \cdots + 138110042112 n^{27})r_{4,5}(n+6).
			\end{aligned}
			$
		}
	\end{align*}
	Accordingly, $R_{4,5}(\boldsymbol{0},z)$ satisfies a differential equation of order $33$ and degree $6$.
\end{theorem}

\begin{corollary}
	As $n\to \infty$,
	\begin{align}
		r_{4,5}(n) \sim C_{4,5}\cdot 80^n n^{-\frac{5}{2}},
	\end{align}
	where the constant $C_{4,5}\approx 0.0353$.
\end{corollary}

\begin{corollary}
	The P\'olya number of the five-dimensional Heracles lattice is approximately
	\begin{align}
		\prob_{4,5}(\boldsymbol{0}) \approx 0.01561.
	\end{align}
\end{corollary}

\section{Cerberus lattices}

Recall that Cerberus lattices are such that $M=3$. Since $M=3$ is odd, we shall look into $\widetilde{r}_{3,N}(n)$ and accordingly $\widetilde{R}_{3,N}(\boldsymbol{0},z)$. In general, we do not have a summation expression as neat as that in Theorem \ref{th:H-sum}. Hence, in this section we have to apply creative telescoping directly to the integral
\begin{align*}
	R_{3,N}(\boldsymbol{0},z)= \frac{1}{\pi^N}\int_{-1}^{1}\cdots \int_{-1}^{1} \frac{\rd x_1\cdots \rd x_N}{\big(1-8z\,\sigma_3(x_1,\ldots,x_N)\big)\sqrt{1-x_1^2}\cdots \sqrt{1-x_N^2}}.
\end{align*} 
Also, we notice that the Cerberus lattice in dimension four is also a Heracles lattice, and it has been studied in Sect.~\ref{sec:H-4D}. So we only consider the five-dimensional case.

\subsection{Case: 5D}\label{sec:C-5D}

When directly applying creative telescoping with recourse to Koutschan's \textbf{CreativeTelescoping} or \textbf{FindCreativeTelescoping} commands to the integral
\begin{align*}
	R_{3,5}(\boldsymbol{0},z)= \frac{1}{\pi^5}\int_{-1}^{1}\cdots \int_{-1}^{1} \frac{\rd x_1\cdots \rd x_5}{\big(1-8z\,\sigma_3(x_1,\ldots,x_5)\big)\sqrt{1-x_1^2}\cdots \sqrt{1-x_5^2}},
\end{align*} 
we notice that the sets of telescopers for the first (with respect to $x_1$) and second (with respect to $x_1$ and $x_2$) integrals can be generated in a few seconds. However, when it comes to the third integral with respect to $x_1$, $x_2$, and $x_3$, our program running on a personal laptop had not yet finished execution after around \emph{ten days} and occupied too much memory (\emph{$\approx 7.9$ Gigabytes} on Day-10). We also tested the same code on an Amazon Server (AWS EC2, instance type: c7i.metal-48xl) and received the telescopers for the third integral in around \emph{five days}. Since such direct calculations for the fourth and fifth integrals should be more complicated, we anticipate that they cannot be completed within a reasonable time.

Fortunately, there is a trick known as \emph{modular reduction} as described in \cite[\S{}3.4]{Kou2009} to handle this issue. In our case, we need to find four independent annihilating operators for the third integral, and each of them is of the form
\begin{align*}
	&\sum P_{\alpha_z,\alpha_{x_4},\alpha_{x_5}}(z,x_4,x_5)\partial_{z}^{\alpha_z}\partial_{x_4}^{\alpha_{x_4}}\partial_{x_5}^{\alpha_{x_5}}\\
	&\qquad+ \partial_{x_3}\cdot \sum \frac{Q_{\beta_z,\beta_{x_4},\beta_{x_5},\mu,\nu}(z,x_4,x_5)\,x_3^\mu}{D_{\beta_z,\beta_{x_4},\beta_{x_5},\nu}(x_3;z,x_4,x_5)} \partial_{x_3}^{\nu}\partial_{z}^{\beta_z}\partial_{x_4}^{\beta_{x_4}}\partial_{x_5}^{\beta_{x_5}},
\end{align*}
where all $P$ and $Q$ are rational functions in $\mathbb{Q}(z,x_4,x_5)$ and all $D$ are polynomials in $\mathbb{Q}(z,x_4,x_5)[x_3]$. The telescopers (i.e., the summation of the $P$-terms in the above) in the four annihilating operators are supported on
\begin{alignat*}{2}
	&\{1, \partial_z, \partial_{x_4}, \partial_{x_5}\},\qquad\qquad
	&&\{1, \partial_z, \partial_{x_5}, \partial_{z}^2, \partial_{z}\partial_{x_5}\},\\
	&\{1, \partial_z, \partial_{x_5}, \partial_{z}^2, \partial_{x_5}^2\},\qquad\qquad
	&&\{1, \partial_z, \partial_{x_5}, \partial_z^2, \partial_z^3\},
\end{alignat*}
respectively.

Here we first notice that the denominators $D$ can be reasonably predicted by the ``\textit{FindSupport}'' mode of Koutschan's \textbf{FindCreativeTelescoping}. Then we are left to construct the rational functions $P$ and $Q$. The trick of modular reduction is that we may first choose a finite set of large primes $p$ and a finite set of interpolation points $(z,x_4,x_5)\in \mathbb{Z}^3$, and then for each choice of $p$ and $(z,x_4,x_5)$ we evaluate $P(z,x_4,x_5)$ and $Q(z,x_4,x_5)$ modulo $p$. This process can be automatically done by the ``\textit{Modular}'' mode of Koutschan's \textbf{FindCreativeTelescoping}. Finally, from the resulting dataset of modular interpolations, we recover our rational functions $P$ and $Q$ by standard rational reconstruction \cite[\S{}5.7]{vzGG2013}.

To avoid ambiguous reconstruction results, our interpolation procedure is executed on a large set of interpolation points, especially for the annihilating operator with telescoper supported on $\{1, \partial_z, \partial_{x_5}, \partial_z^2, \partial_z^3\}$. For this particular case, each $z$, $x_4$ and $x_5$ ranges over $\mathtt{100}$ integers. We also choose $\mathtt{10}$ large primes of value around $2^{32}$ to conduct the modular reduction. Hence in total, we have $\mathtt{10*100^3}$ (\emph{ten million}!) interpolations to be evaluated. These computations were performed in \emph{parallel} on an Amazon Server (AWS EC2, instance type: c7i.metal-48xl, \emph{\textsl{192}-kernel}), and it took us around \emph{two hours} to process the interpolation evaluations and another \emph{three hours} for rational reconstruction.

After constructing the annihilating operators for the third integral (stored in a plain text file of \emph{$1.35$ Megabytes}), we repeat the modular reduction process for the fourth and fifth integrals (\emph{$5.41$ Megabytes} and \emph{$32.5$ Megabytes}, respectively). Finally, we arrive at the following result.

\begin{theorem}\label{th:M3N5}
	The function $\widetilde{R}_{3,5}(\boldsymbol{0},z)$ satisfies a differential equation of order $14$ and degree $55$. The differential equation is of the form
	\begin{align*}
		\scalebox{0.85}{%
			$
			\begin{aligned}
				&(520229399\ldots552000000 z^2 + \cdots - 107399360\ldots000000000 z^{55})\\
				&+(125607529\ldots151840000 z + \cdots - 991320892\ldots000000000 z^{55})\vartheta\\
				&+(108022475\ldots505824000 z + \cdots -407803160\ldots000000000 z^{55})\vartheta^{2}\\
				&-(272568338\ldots194928000 z + \cdots +993711876\ldots000000000 z^{55})\vartheta^{3}\\
				&-(501189743\ldots173558000 z + \cdots +160535663\ldots000000000 z^{55})\vartheta^{4}\\
				&-(471028235\ldots681940000 + \cdots +182119420\ldots000000000 z^{55})\vartheta^{5}\\
				&+(660224576\ldots341859000 + \cdots -149746997\ldots000000000 z^{55})\vartheta^{6}\\
				&-(394721661\ldots746572000 + \cdots +907166108\ldots000000000 z^{55})\vartheta^{7}\\
				&+(131991924\ldots722961750 + \cdots -407064227\ldots000000000 z^{55})\vartheta^{8}\\
				&-(271983478\ldots244204500 + \cdots +134692988\ldots000000000 z^{55})\vartheta^{9}\\
				&+(358216973\ldots115370000 + \cdots -323639321\ldots000000000 z^{55})\vartheta^{10}\\
				&-(301728911\ldots358715500 + \cdots +547805151\ldots000000000 z^{55})\vartheta^{11}\\
				&+(156905393\ldots200238250 + \cdots -617749336\ldots000000000 z^{55})\vartheta^{12}\\
				&-(457839444\ldots684568000 + \cdots +415634712\ldots000000000 z^{55})\vartheta^{13}\\
				&+(572299306\ldots335571000 + \cdots -125949912\ldots000000000 z^{55})\vartheta^{14}.
			\end{aligned}
			$
		}
	\end{align*}
	Accordingly, $\widetilde{r}_{3,5}(n)$ satisfies a recurrence of order $55$.
\end{theorem}

A surprising fact about the above differential equation is that the magnitude of its coefficients becomes dramatically huge compared with those in Theorems \ref{th:M3N4} and \ref{th:M4N5-ODE}. In fact, the largest among these coefficients has $203$ digits!

By consulting the \textsf{Guess} package, the order of our differential equation for $\widetilde{R}_{3,5}(\boldsymbol{0},z)$ is possibly minimal. We also want to find a recurrence satisfied by $\widetilde{r}_{3,5}(n)$ with the minimal order in view of the procedure presented at the end of Sect.~\ref{sec:ODE-REC}. To embark on this task, we need to generate a list of $\widetilde{r}_{3,5}(n)$ of a considerable length. With recourse to Theorem \ref{th:M3N5}, the first $55$ terms are sufficient to produce our desired list of length, for example, $900$. Meanwhile, these initial values can be derived by the differential equation in Theorem \ref{th:M3N5}:
\begin{align*}
	\big\{\widetilde{r}_{3,5}(n)\big\}_{n=0}^{54}&=\{1, 80, 71280, 174723200, 573097798000, 2167896636622080,\ldots,\\
	&\quad\,\,\,\,\underbrace{20198893220533155882232776\ldots330349744141730893312000}_{\text{$200$ digits!}}\}.
\end{align*}
Here, the last value $\widetilde{r}_{3,5}(54)$ has $200$ digits! Finally, the following results follow according to the same reasoning as that for the two Heracles lattices in Subsects.~\ref{sec:H-4D} and \ref{sec:H-5D}.

\begin{theorem}
	The sequence $\widetilde{r}_{3,5}(n)$ satisfies a recurrence of order $8$. The recurrence is of the form
	\begin{align*}
		\scalebox{0.85}{%
			$
			\begin{aligned}
				0&=(322911616\ldots000000000 + \cdots + 192276259\ldots642240000 n^{61})\widetilde{r}_{3,5}(n)\\
				&-(444007451\ldots120000000 + \cdots + 255667339\ldots105241600 n^{61})\widetilde{r}_{3,5}(n+1)\\
				&+(441030057\ldots472000000 + \cdots + 679521248\ldots629222400 n^{61})\widetilde{r}_{3,5}(n+2)\\
				&-(123290773\ldots436800000 + \cdots + 700480775\ldots712614400 n^{61})\widetilde{r}_{3,5}(n+3)\\
				&+(137273738\ldots934400000 + \cdots + 346762466\ldots298675200 n^{61})\widetilde{r}_{3,5}(n+4)\\
				&-(660847461\ldots160000000 + \cdots + 844677552\ldots970444800 n^{61})\widetilde{r}_{3,5}(n+5)\\
				&+(120581152\ldots856000000 + \cdots + 863236627\ldots610124800 n^{61})\widetilde{r}_{3,5}(n+6)\\
				&-(294556299\ldots528000000 + \cdots + 142224074\ldots019507200 n^{61})\widetilde{r}_{3,5}(n+7)\\
				&+(120558800\ldots440000000 + \cdots + 416932677\ldots459545600 n^{61})\widetilde{r}_{3,5}(n+8).
			\end{aligned}
			$
		}
	\end{align*}
	Accordingly, $\widetilde{R}_{3,5}(\boldsymbol{0},z)$ satisfies a differential equation of order $69$ and degree $8$.
\end{theorem}

\begin{corollary}
	As $n\to \infty$,
	\begin{align}
		\widetilde{r}_{3,5}(n) = r_{3,5}(2n) \sim \widetilde{C}_{3,5}\cdot 6400^n n^{-\frac{5}{2}},
	\end{align}
	where the constant $\widetilde{C}_{3,5}\approx 0.0128$.
\end{corollary}

\begin{corollary}
	The P\'olya number of the five-dimensional Cerberus lattice is approximately
	\begin{align}
		\prob_{3,5}(\boldsymbol{0}) \approx 0.01581.
	\end{align}
\end{corollary}

\section{Conclusion}\label{sec:conclusion}

\begin{table}[ht!]
	\def\arraystretch{1.5}
	\centering
	\caption{$M$-Headed lattice in dimension $N$}\label{tab:lattice}
	\begin{tabular}{cc|>{\centering}m{2.1cm}>{\centering}m{2.5cm}|c}
		\multicolumn{5}{l}{{\footnotesize \textsuperscript{$\dagger$}Entries with a gray background are newly computed in this work.}}\\
		\hline
		\multirow{2}{*}{$N$} & \multirow{2}{*}{$M$} & \multicolumn{2}{c|}{Ddifferential Equation} & \multirow{2}{*}{P\'olya Number} \\
		& & {\footnotesize Order} & {\footnotesize Degree} & \\
		\hline
		$1$& $1$ & $1$ & $2$ & $1$ \\\hline
		$2$& $1$ & $2$ & $2$ & $1$ \\\arrayrulecolor{gray!50}\hline\arrayrulecolor{black}
		$2$& $2$ & $2$ & $2$ & $1$ \\\hline
		$3$& $1$ & $3$ & $4$ & $\begin{gathered}
			1-2^{4}(\sqrt{3}+1)\pi^{3}\Gamma(\tfrac{1}{24})^{-2}\Gamma(\tfrac{11}{24})^{-2}\\
			\scalebox{0.8}{$(\approx 0.34054)$}
		\end{gathered}$ {\rule{0pt}{1.75\normalbaselineskip}}\\\arrayrulecolor{gray!60}\hline\arrayrulecolor{black}
		$3$& $2$ & $3$ & $3$ & $\begin{gathered}
			1-2^{\frac{14}{3}}3^{-2}\pi^{4}\Gamma(\tfrac{1}{3})^{-6}\\
			\scalebox{0.8}{$(\approx 0.25632)$}
		\end{gathered}$ {\rule{0pt}{1.75\normalbaselineskip}}\\\arrayrulecolor{gray!60}\hline\arrayrulecolor{black}
		$3$& $3$ & $3$ & $2$ & $\begin{gathered}
			1-2^{2}\pi^{3}\Gamma(\tfrac{1}{4})^{-4}\\
			\scalebox{0.8}{$(\approx 0.28223)$}
		\end{gathered}$ {\rule{0pt}{1.75\normalbaselineskip}}\\\hline
		$4$& $1$ & $4$ & $4$ & $\approx 0.19313$ \\\arrayrulecolor{gray!60}\hline\arrayrulecolor{black}
		& & $4$ & $7$ &  \\
		\multirow{-2}{*}{$4$}& \multirow{-2}{*}{$2$} & \cellcolor{gray!20}$11$ & \cellcolor{gray!20}$5$ & \multirow{-2}{*}{$\approx 0.09571$} \\\arrayrulecolor{gray!60}\hline\arrayrulecolor{black}
		& & \cellcolor{gray!20}$8$ & \cellcolor{gray!20}$32$ & \cellcolor{gray!20} \\
		\multirow{-2}{*}{$4$}& \multirow{-2}{*}{$3$} & \cellcolor{gray!20}$24$ & \cellcolor{gray!20}$8$ & \cellcolor{gray!20}\multirow{-2}{*}{$\approx 0.04332$} \\\arrayrulecolor{gray!60}\hline\arrayrulecolor{black}
		$4$& $4$ & $4$ & $2$ & $\approx 0.10605$ \\\hline
		$5$& $1$ & $5$ & $6$ & $\approx 0.13517$ \\\arrayrulecolor{gray!60}\hline\arrayrulecolor{black}
		& & $6$ & $13$ &\\
		\multirow{-2}{*}{$5$}& \multirow{-2}{*}{$2$} & \cellcolor{gray!20}$19$ & \cellcolor{gray!20}$7$ & \multirow{-2}{*}{$\approx 0.04657$} \\\arrayrulecolor{gray!60}\hline\arrayrulecolor{black}
		& & \cellcolor{gray!20}$14$ & \cellcolor{gray!20}$110$ & \cellcolor{gray!20} \\
		\multirow{-2}{*}{$5$}& \multirow{-2}{*}{$3$} & \cellcolor{gray!20}$69$ & \cellcolor{gray!20}$16$ & \cellcolor{gray!20}\multirow{-2}{*}{$\approx 0.01581$} \\\arrayrulecolor{gray!60}\hline\arrayrulecolor{black}
		& & \cellcolor{gray!20}$9$ & \cellcolor{gray!20}$24$ & \cellcolor{gray!20} \\
		\multirow{-2}{*}{$5$}& \multirow{-2}{*}{$4$} & \cellcolor{gray!20}$33$ & \cellcolor{gray!20}$6$ & \cellcolor{gray!20}\multirow{-2}{*}{$\approx 0.01561$} \\\arrayrulecolor{gray!60}\hline\arrayrulecolor{black}
		$5$& $5$ & $5$ & $2$ & $\approx 0.04473$ \\
		\hline
	\end{tabular}
\end{table}

In this work, we complete the investigation of all multi-headed lattices in dimensions no larger than five. Some important statistics such as the orders and degrees of the corresponding differential equations and the P\'olya numbers for these lattices are recorded in Table \ref{tab:lattice}. Apart from the three new results in our work, namely, the Heracles lattices ($M=N-1$) in dimensions $N=4$ and $5$ and the Cerberus lattice ($M=3$) in dimension $N=5$, as well as the results of Koutschan \cite{Kou2013} who confirmed experimental discoveries on the Orthrus lattices ($M=2$) in dimensions $N=4$ due to Guttmann \cite{Gut2009} and $5$ due to Broadhurst \cite{Bro2009}, most information in Table \ref{tab:lattice} can be found in Guttmann's survey \cite{Gut2010}.

It is notable that in \cite{Gut2010}, what Guttmann recorded were the differential equations for lattice Green functions $P(\boldsymbol{0},z^{\frac{1}{2}})$ rather than $P(\boldsymbol{0},z)$ if odd powers of $z$ vanish in the series expansion of $P(\boldsymbol{0},z)$, that is, in the cases where $M$ is odd or $M$ equals $N$. Hence, the degrees of these equations should be doubled to match the values in Table \ref{tab:lattice}. The same situation also happens to the two Cerberus lattices in our Theorems \ref{th:M3N4} and \ref{th:M3N5}. We also remind the reader that Koutschan's differential equations in \cite{Kou2013} are with respect to the usual derivation operator $D=\frac{\rd}{\rd z}$ instead of our $\vartheta=z\frac{\rd}{\rd z}$, so it requires some extra work to recover the corresponding degrees in Table \ref{tab:lattice}.

Meanwhile, as pointed out at the end of Sect.~\ref{sec:ODE-REC}, we intend to find differential equations for Green functions with either the smallest order or the smallest degree. We have applied the procedure therein to go through all multi-headed lattices in dimensions no larger than five. All known differential equations are of the minimal order, and there are only two exceptional cases with a larger degree, both for Orthrus lattices. We report the two new results as follows with their computer-assisted proofs given in \cite{CCJ2024}.

\begin{theorem}[Four-dimensional Orthrus lattice]\label{th:M2N4}
	The sequence $r_{2,4}(n)$ satisfies a recurrence of order $5$. The recurrence is of the form
	\begin{align*}
		\scalebox{0.85}{%
			$
			\begin{aligned}
				0&=(287649792 + \cdots + 967680 n^{6})r_{2,4}(n)\\
				&+ (708258816 + \cdots + 725760 n^{6})r_{2,4}(n+1)\\
				&+(379157760 + \cdots + 176400 n^{6})r_{2,4}(n+2)\\
				&+ (55519056 + \cdots + 14000 n^{6})r_{2,4}(n+3)\\
				&-(638976 + \cdots + 105 n^{6})r_{2,4}(n+4)\\
				&-(345000 + \cdots + 35 n^{6})r_{2,4}(n+5).
			\end{aligned}
			$
		}
	\end{align*}
	Accordingly, $R_{2,4}(\boldsymbol{0},z)$ satisfies a differential equation of order $11$ and degree $5$.
\end{theorem}

\begin{theorem}[Five-dimensional Orthrus lattice]\label{th:M2N5}
	The sequence $r_{2,5}(n)$ satisfies a recurrence of order $7$. The recurrence is of the form
	\begin{align*}
		\scalebox{0.85}{%
			$
			\begin{aligned}
				0&=(42140738676326400000 + \cdots + 3986266521600 n^{12})r_{2,5}(n)\\
				&+ (118427858324029440000 + \cdots + 3056137666560 n^{12})r_{2,5}(n+1)\\
				&+(62676619662919680000 + \cdots + 668530114560 n^{12})r_{2,5}(n+2)\\
				&+ (1794185247360768000 + \cdots + 7266631680 n^{12})r_{2,5}(n+3)\\
				&-(3522851180688416000 + \cdots + 11354112000 n^{12})r_{2,5}(n+4)\\
				&-(458904717778020000 + \cdots + 908328960 n^{12})r_{2,5}(n+5)\\
				&-(1106658753555600 + \cdots + 1013760 n^{12})r_{2,5}(n+6)\\
				&+(836209651013100 + \cdots + 760320 n^{12})r_{2,5}(n+7).
			\end{aligned}
			$
		}
	\end{align*}
	Accordingly, $R_{2,5}(\boldsymbol{0},z)$ satisfies a differential equation of order $19$ and degree $7$.
\end{theorem}

Finally, as in the same boat with previous work such as \cite{Kou2013}, the only obstacle that prevents us from marching on higher-dimensional cases is the exhausting computations for the differential equations based on creative telescoping or other means in a computer algebra system. Hence, we are not very optimistic about certifying these equations in a rigorous way, though it is still hopeful to take advantage of the heuristic methodology in \cite{HKMZ2016} to receive convincing information about higher-dimensional lattices.

\subsection*{Acknowledgements}

LJ is grateful to Christoph Koutschan for conversations regarding the modular reduction trick for creative telescoping.

\bibliographystyle{amsplain}

\end{document}